          \documentclass{article}
\usepackage[paperwidth=7in, paperheight=10in, margin=.875in]{geometry}
\usepackage{amsmath}
\usepackage{setspace,graphicx}     
\usepackage{amsfonts}
\usepackage{subfig}
\usepackage{latexsym}
\usepackage{amssymb}

\usepackage{amsthm}
\usepackage{stackengine}
\usepackage{setspace}
\usepackage{fancyhdr}
\usepackage{mathrsfs}
\usepackage{framed}
\usepackage{bm}
\usepackage{dsfont}
\usepackage{comment}
\usepackage{color}
\definecolor{darkgreen}{rgb}{0,0.7,0}

\definecolor{darkblue}{rgb}{0,0,0.7}

\DeclareMathOperator*{\argmin}{argmin}

\newcommand{\R}{\mathbb{R}}

\newcommand{\D}{\overline{D}}

\newcommand{\N}{\mathbb{N}}

\newcommand{\Ss}{\mathcal{S}}

\renewcommand{\S}{\mathcal{S}}

\newcommand{\Xk}{X^{(k)}}
\newcommand{\Xone}{X^{(1)}}
\newcommand{\Xii}{X^{(i)}}
\newcommand{\Xj}{X^{(j)}}
\newcommand{\Ximone}{X^{(i-1)}}

\newcommand{\tauone}{\tau^{(1)}}
\newcommand{\taui}{\tau^{(i)}}
\newcommand{\tauj}{\tau^{(j)}}
\newcommand{\tauimone}{\tau^{(i-1)}}
\newcommand{\taujmone}{\tau^{(j-1)}}
\newcommand{\taukmone}{\tau^{(k-1)}}

\newcommand{\envelope}{(\raisebox{-.5pt}{\scalebox{1.45}{\Letter}}\kern-1.7pt)}

\DeclareMathOperator*{\kmin}{k-\,min}

\definecolor{mygreen}{rgb}{0.1,0.75,0.2}

\newcommand{\nc}{\normalcolor}

\newcommand{\Prob}{\mathbb{P}}

\newcommand{\gto}{\xrightarrow{\Gamma}}

            \newtheorem{thm}{Theorem}[section]
          \newtheorem{prop}[thm]{Proposition}
           
          \newtheorem{lem}[thm]{Lemma}
          \newtheorem{cor}[thm]{Corollary}
          \newtheorem{remark}[thm]{Remark}

          \newtheorem{defn}[thm]{Definition}

          \newtheorem{rem}[thm]{Remark}

\begin{document}
	   \title{Spatial extreme values: variational techniques and stochastic integrals}

\author{Nicol\'as Garc\'ia Trillos\thanks{Department of Statistics, University of Wisconsin.} \and{} Ryan Murray\thanks{Department of Mathematics, Pennsylvania State University} \and{} Daniel Sanz-Alonso\thanks{Department of Statistics, University of Chicago.}}




         \pagestyle{myheadings} \markboth{Title}{} \maketitle

          \begin{abstract}
This work employs variational techniques to revisit and expand the construction and analysis of extreme value processes. These techniques permit a novel study of spatial statistics of the location of minimizing events. We develop integral formulas for computing statistics of spatially-biased extremal events, 
and show that they are analogous to stochastic integrals in the setting of standard stochastic processes.
We also establish an asymptotic result in the spirit of the Fisher-Tippett-Gnedenko theory for a broader class of extremal events and discuss some applications of our results. 
          \end{abstract}

\section{Introduction}
Extreme value theory is the branch of probability theory and statistics concerned with the study of extreme deviations from the median behavior. A fundamental problem is to characterize the asymptotic distribution of
\begin{displaymath}
  M_n = a_n\min_{i=1 \dots n} (X_i-b_n),
\end{displaymath}
where the $X_i$ are i.i.d. random variables. A complete account of the possible limit distributions of this type of process is given by the Fisher-Tippett-Gnedenko theory (see e.g. \cite{resnick2013extreme}), which also explains the role of the normalizing sequences $
\{a_n\}$ and $\{b_n\}$. For instance if the $X_i$ are strictly positive with unit density at zero and  $a_n = n, \, b_n = 0$, then the $M_n$ converge in law to the exponential distribution with unit rate. This classic theory and some related stochastic processes will be reviewed in Subection \ref{sec:related-work}.
In many modern settings it is desirable to not only understand the distribution of minima, but also of the index (or ``spatial location'') of minimizers; this paper investigates some questions arising in such settings.
Suppose, as an illustration, that a user sends a request to a group of $n$ spatially-distributed servers with distinct locations $x_1, \dots, x_n \in D \subset \R^d$. The user and the parties running the servers will certainly be interested in waiting time for the first response to occur, but also will likely be interested in identifying the location of the first responding server. In order to model such a system, we will suppose that the response time of the servers is given by 
\begin{equation}
  f_n(x) := \begin{cases}  \frac{ n}{\lambda(x_i)}\xi_i + g(x_i) , &\text{ for } x = x_i, \quad  i = 1,\dots, n,\\ +\infty, &\text{ otherwise,}\end{cases}
\end{equation}
where $g(x_i)$ represents a deterministic communication latency (i.e. the amount of time it takes the message to reach the $i$-th server) that is \emph{not} necessarily spatially uniform, the $\xi_i$ are i.i.d. exponential variables, and the $\lambda(x_i)$ represents the processing speed of the server at location $x_i$ (which could be related to either hardware or workload). We have chosen the scaling $n$ for the random exponential term so that in the large $n$ limit it has comparable size to the deterministic latency. In a way that will be made precise, this scaling ensures that, for large $n,$ the overall processing rate of all servers is of order $1$ and given by the function $\lambda$. The latency imposes a relatively weak deterministic spatial correlation structure and models a scenario where the processing power needed for the request is high and the communication overhead is low. The locations $x_i$ of the servers will be assumed to be distributed according to a probability density $\rho$ in a bounded domain $D\subset \R^d$ so that, in general, the locations will not be uniformly distributed. 
In this framework, it is natural to study the distribution of the first response time, along with the distribution of the \emph{location} of the first server making a response. 

A first contribution of this paper is to provide integral representation formulas, which to our knowledge are novel, for the distribution of mins and argmins. These formulas may be of use in various inference problems, some of which will be outlined in Section \ref{sec:extensions}. We will be mostly interested in the limit problem as $n \to \infty$ and a second contribution is to introduce an appropriate analytical setting in which to study such a limit using modern tools from the calculus of variations \cite{DalMaso}. The response time functions $f_n$ will be viewed as  random variables with values in the space of lower semi-continuous functions $\mathcal{S}(D)$ (a precise definition of this space, and the associated topology of $\Gamma$-convergence, will be given in Section \ref{sec:topology}). Intuitively, $\mathcal{S}(D)$ is a natural choice of function space, as it permits the evaluation of minima and minimizers both in the discrete case with finite $n$ and in the limiting case, $n \to \infty$.  Moreover, the functions $f_n$ above are immediately lower semi-continuous, as they are defined to be $+\infty$ except at the $x_i$. Finally, a third contribution of this paper is to draw several analogies between extreme value processes and standard stochastic processes. We show that, if $g\equiv 0,$ the large $n$ distribution of mins and argmins of the response times $f_n$ is governed by the distribution of mins and argmins of a limiting extreme value process $W_\lambda$ that can be seen as an analog of Brownian motion. The process $W_\lambda$  has been studied before, but here our focus is on the spatial statistics of minimizers, which differs from much of the previous literature --see Subsection \ref{sec:related-work}. In the case of general $g$, evaluating local minima is analogous to computing stochastic integrals of $g,$ with extreme value processes playing the role of Brownian motion. 

The remainder of the introduction goes as follows. In Subsection \ref{ssec:Main theorems} we state our main results. We review related work in Subsection \ref{sec:related-work}, and applications and future lines of research in Subsection \ref{sec:extensions}. We close with an outline of the paper in \ref{ssec:outline}.

\subsection{Main results}\label{ssec:Main theorems}

The main results of this paper are concerned with the random process $W_\lambda$, which will serve as a limit of the process described above in the case $g\equiv 0$. Here and throughout $D$ denotes a bounded domain in $\R^d$, $\bar{D}$ its closure, and $\Ss(\overline{D})$ the space of lower semi-continuous function on $\overline{D}$ (see Section \ref{sec:prelim} for more details on this space).

\begin{defn}\label{def:W}
Let $\lambda : \overline{D} \rightarrow (0, \infty)$ be a continuous function. We define $W_\lambda$ to be the random object with values in $\Ss(\overline{D})$ satisfying the following properties:
	\begin{enumerate}
		\item For every  closed set $C$ the random variable $\min\limits_{x \in C\cap \overline{D}} W_\lambda(x)$ is exponentially distributed with rate
		\[ \lambda_C := \int_{C \cap \overline{D}} \lambda(x) dx.\]
		\item For any finite collection of disjoint and closed sets $C_1, \dots, C_k$, the random variables 
		\[ \min_{x \in C_1} W_\lambda(x) , \dots, \min_{x \in C_k} W_\lambda(x) \]
		are independent. 
	\end{enumerate}
\end{defn}
We will show in Propositions \ref{prop:W-unique} and \ref{ConvergenceGamma} that in fact the random process $W_\lambda$ exists and is well defined. Furthermore, we will show in Proposition \ref{prop:kmins} that after adding a deterministic function $g$ to $W_\lambda$, the resulting random function has well-defined, unique  $k$-argmins (as defined in Definition \ref{def:k-mins}). Our first main result establishes a representation formula for the joint distribution of the first $k$-argmins. 

\begin{thm}
\label{thm:main}
Let $g \in \Ss(\overline{D}) $ (see Section \ref{sec:topology}) and let $W_\lambda$ be as in Definition \ref{def:W}. Then, (with probability 1) the random function $W_\lambda + g$ has well-defined, unique $k$-argmins for every $k \in \N$ (see Definition \ref{def:k-mins}). In addition, the joint distribution of the first $k$-argmins $(\Xone, \dots, \Xk)$ is given by the density function:
\begin{align}
\label{eqn:argmin-dist}
\begin{split}
  \frac{\Prob(\Xone \in dx_1, \dots, \Xk \in dx_k )}{dx_1 \dots dx_k} = &\int_{\R^{k-1}} \Phi( x_k,(g-r_{k-1})^+) \Psi(x_1,r_1,g) \\
  &\times \prod_{j=2}^{k-1} \Psi(x_j,r_j-r_{j-1},(g-r_{j-1})^+) \,dr_1\dots\,dr_{k-1},
  \end{split}
\end{align}
with
\begin{align}
\begin{split}
  \Phi(z,h) &:= \int_{h(z)}^\infty  \exp\left( -\bar \lambda\int_{-\infty}^t H_{\lambda,h}(s) \,ds \right)\,dt,\\
  \Psi(z,t,h) &:= \chi_{t > h(z)} \exp\left( -\bar \lambda\int_{-\infty}^t H_{\lambda,h}(s) \,ds \right),\nc\\
  H_{\lambda,h}(s) &:= \mu_\lambda(\{x \in D : h(x) \leq s\}),\label{eqn:H-def} \\
  \bar\lambda &:= \int_D \lambda(x) \,dx, \\
  \mu_\lambda(E) &:= \frac{1}{\bar \lambda}\int_E \lambda(x) \,dx.
  \end{split}
\end{align}
\end{thm}

Our second main result establishes that $W_\lambda + g$ serves as an appropriate limiting process for a wide class of discrete processes, including the $f_n$ above.  
\begin{thm}
\label{thm:main2}
Suppose that  $x_1, \dots, x_n, \dots$ is a sequence of points in $D$ for which the empirical measures 
\[  \frac{1}{n}\sum_{i=1}^n \delta_{x_i} \]
converge weakly, as $n \rightarrow \infty$, towards the measure $\rho(x) dx$, where $\rho: D \rightarrow \R$ is a continuous density function. Let $\lambda: \overline{D} \rightarrow \R$ be a continuous function and let $g: \overline{D}\rightarrow \R$. Let  $\xi_1, \dots, \xi_n, \dots$  be independent, non-negative random variables (not necessarily identically distributed) whose cumulative distribution functions $F_i$ satisfy 
\begin{equation} F_i(t) = t+ O(t^2) \label{eqn:first-order-dist}    \end{equation}
near zero (uniformly in $i$). Define the random function $f_n$ by
\[  f_n(x) = \begin{cases}  \frac{n}{\lambda(x_i)} \xi_i  + g(x_i), & \text{if } x=x_i , \quad i=1, \dots, n,\\ +\infty, & \text{otherwise}.  \end{cases}  \]
Then, the joint distribution of the first $k$ argmins of $f_n$ converges, as $n \rightarrow \infty$, towards the joint distribution of the first $k$ argmins of $W_{\tilde{\lambda}}+g$, where
\[ \tilde{\lambda}(x) = \lambda(x) \rho(x), \quad x \in D. \]
\end{thm}
The proofs of Theorems \ref{thm:main} and \ref{thm:main2} will be given in Section \ref{sec:main-proofs}.
An immediate corollary of these theorems is the following:
\begin{cor}
  In the setting of Theorem \ref{thm:main2}, the joint distribution of the first $k$ argmins of $f_n$ converges, as $n \rightarrow \infty$, towards the distribution given by \eqref{eqn:argmin-dist} with $\lambda$ replaced by  $\tilde{\lambda}(x) = \lambda(x) \rho(x).$
\label{cor:main}
\end{cor}

\begin{remark}\label{rem:alpha-stable-analog}
  In most  proofs below the random variables $\xi_i$ are assumed to be exponentially distributed with unit rate and the points $x_i$ to be uniformly distributed. This assumption can be made without loss of generality, as the rates of the $\xi_i$ and fluctuations in $\rho$ may be absorbed into $\lambda$. Moreover, the assumption that the $\xi_i$ are exponential may be relaxed in the limit as $n \to \infty$ (see e.g. Theorem \ref{thm:main2}). Second, the Fisher-Tippett-Gnedenko theory provides a broader class of possible limiting distributions (analogous to $\alpha$-stable distributions in the classical central limit theory), which would not require such a specific behavior of the $F_i$ near zero. We believe that it is of interest to extend our results to that broader family, but such an extension is beyond the scope of this work.
\end{remark}

\subsection{Related work}\label{sec:related-work}
Extreme value theory and extremal processes is a classical branch of probability and statistics. While we do not attempt to give an exhaustive account of the field, we will highlight some issues that are relevant to our work. We will also draw parallels with other limit theorems and stochastic processes.

One of the first questions in probability theory is to understand the limits of combinations of independent random variables. In the context of the central limit theorem, one considers
\begin{displaymath}
 \lim_{n \to \infty} {n}^{-1/2} \left(\sum_{i=1}^n X_i -  n \mathbb{E}(X_i)\right).
\end{displaymath}
The central limit theorem states that, for i.i.d. $X_i$ with finite variance, the limit is a normally-distributed random variable. In the context of extreme values, given a sequence $\{X_i\}$ of i.i.d. variables one similarly considers 
\begin{displaymath}
  \lim_{n \to \infty} a_n\left(\min_{i=1\dots n} X_i - b_n\right).
\end{displaymath}
Here $a_n$ and $b_n$ are normalizing constants, analogous to $n^{-1/2}$ and $n \mathbb{E}(X_i)$ in the central limit theorem case. The Fisher-Tippett-Gnedenko Theorem \cite{Gnedenko} completely specifies possible limits for this process. A detailed description of this theory can be found in \cite{Leadbetter} or \cite{resnick2013extreme}. In this paper we focus on the case where $X_i$ is positive, and has a density function which is positive at zero.

Subsequently, various authors \cite{Dwass1964} \cite{Lamperti} studied the distribution of the ``k-records''
\begin{displaymath}
  M_n^k = a_n\left(\kmin_{i=1\dots n} X_i - b_n\right),
\end{displaymath}
where $\kmin\limits_{i=1\dots n}$ denotes the value of the $k$-th smallest element in the collection. Studying $k$-th mins is an important branch of extreme value theory that we pursue and extend here by characterizing the spatial distribution of the $k$-th minima. 

Returning to the discussion of first mins, a natural object is the rescaled process 
\begin{displaymath}
  M_n(t) = a_n\left(\min_{i < n t} X_i - b_n\right).
\end{displaymath}
that tracks mins over time. Taking the limit $\lim_{n \to \infty} M_n(t) =: M(t)$ yields so-called \emph{extremal processes} \cite{resnick2013extreme}. The construction of such processes is completely analogous to the construction of Brownian motion by L\'evy, where sums are replaced by mins. 

In standard stochastic processes, a central concept is the family of \emph{ stable stochastic processes}, which are invariant under linear combinations. In the context of extreme values one instead considers the family of \emph{max-stable processes}, which are stochastic processes that satisfy
\begin{displaymath}
  \max_{i=1\dots r} M^{(i)}(t) \stackrel{d}{=} \, r M(t),
\end{displaymath}
where the $M^{(i)}$ are independent copies of $M$ and $\stackrel{d}{=}$ denotes equality in distribution. A significant literature studies these processes by means of Poisson processes on the plane \cite{Pickands71}, and gives a spectral representation of these processes \cite{deHaan-1984-Spectral}. Min-stable processes (namely processes $M_t$ for which $\frac{1}{M_t}$ is max-stable) were studied in \cite{deHaan-Pickands}. Again, much of this literature focuses on the distribution of minima and not on the locations at which minima occur.

The previous objects can be easily generalized by evaluating the minimum value over some Borel set $A \subset \R$,
\begin{displaymath}
  M_n(A) = a_n \left(\min_{\frac{i}{n} \in A} X_i - b_n\right).
\end{displaymath}
Note that letting $A=[0,t)$ recovers the previously-defined process $M_n(t)$. Again,  one may then consider the limit
  \begin{displaymath}
    M(A) := \lim_{n \to \infty} M_n(A).
  \end{displaymath}
The set functions $M_n$,  $M$ are called \emph{inf-measures}. In the context of maximization the analog is known as a \emph{sup-measure}. 

In this paper we work at the level of lower semi-continuous functions, that is, we study directly the functions $\lim_{n \to \infty} n X_{\frac{i}{n}}$ as opposed to the minimum value that they take. This seems more natural here, as it allows us to study simultaneously both minimizers and minima. The approach via lower semi-continuous functions can be shown to be equivalent to that of inf-measures. Namely, given a inf-measure $M,$ we can define the \emph{inf-derivative} of that measure as
\begin{displaymath}
  d^{\wedge}M (x) := \inf_{G : x \in G} M(G),
\end{displaymath}
where $G$ is allowed to vary across Borel sets. Here $d^{\wedge} M$ will be a lower semi-continuous function. Similarly, given a lower semi-continuous function $f$ we can define the \emph{inf-integral} of the function via
\begin{displaymath}
  f^{\wedge} (G) := \inf_{x \in G} f(x).
\end{displaymath}
Again, $G$ here is any Borel set. It can be shown that $f^{\wedge}$ is then an inf-measure. Hence one can develop the theory either in terms of inf-measures or in terms of lower semi-continuous functions.

A mathematically-sophisticated development of sup-measures, as well as a detailed account of connections with probability, optimization and analysis literature, is given in the excellent article \cite{VervaatTopology}. Unfortunately, Vervaat passed away before his work was published, and hence his work is only published somewhat obscurely.

The papers \cite{Resnick-Roy-1991} and \cite{Resnick-Roy-1992} extend and apply many of the ideas in \cite{VervaatTopology} to construct general random upper semi-continuous functions. These works are closely related to ours in that they construct a version of the process $W_\lambda$. However, they are strongly connected to the framework of choice optimization.

Various authors have addressed spatial effects in different contexts. For example, \cite{deHaan-Lin} studies the probability that the maximum of some sequence of random functions $\xi_i(t)$ exceeds a deterministic function $f$, with $t \in [0,1]$. This is similar to our framework in that they permit spatially-inhomogeneous shifts, but they do not study the distribution of the location of extremal events. Statistical estimation of spatial extremes in the context of max-stable processes has been used to study various geophysical processes \cite{Smith-1990}\cite{Davidson-2012}. These works focus primarily on using spectral representations of max stable processes to tackle challenging spatial statistical problems. 

\subsubsection{Analogue with stochastic integrals}

The discussion above suggests that for most of the basic developments of classical stochastic processes there is an analogue in terms of extreme values. In all cases the main difference is that taking sums of some underlying variables is replaced by taking the minima. In this light, one can see the theory of extreme values as a version of stochastic processes where the algebraic operator ``$+$'' is replaced with ``$\min$''. Extending this analogy, a family of algebraic operations similar to the ring $(+,\times)$, may be obtained by considering the operations $(\min, +)$. 
The resulting algebraic structure\footnote{Technically this is a semi-ring and not a proper ring, because the $\min$ operator is not invertible. However, this will not be important for our purposes.} $(\min, +)$ is known as a \emph{tropical algebra}.

Our discussion of extremal processes above  has not made use of ``multiplication'' of processes. However, this is critical for the construction of stochastic integrals. Recall that with classical Brownian motion one defines
\begin{displaymath}
  \int_0^t H dB = \lim_{n\to \infty} \sum_{i=1}^n H_{t_{i-1}} (B_{t_i} - B_{t_{i-1}}),
\end{displaymath}
where $t_i$ describe partitions of $[0,t]$.

Since in studying extreme values we are replacing sums with mins and multiplication with addition, we thus ought to define our extremal stochastic integrals via
\begin{displaymath}
  ``\int_0^t g dW_\lambda "= \lim_{n \to \infty} \min_{i < n t} g\left(\frac{i}{n}\right) + n  X_{\frac{i}{n}}.
  \end{displaymath}
In this paper we study these processes and  give explicit formulae for the distribution of their first $k$ mins and argmins.

%
%

\subsection{Applications and extensions}\label{sec:extensions}
There are several extensions and special cases of interest that can be derived from our main results. In this section we briefly describe a few promising examples.

\subsubsection{Bayesian estimation via $k$-argmins}

Suppose that the server problem in the introduction is repeated many times. That is, one sequentially sends requests to servers, and observes which were the first $k$ servers to respond, as well as their response times. One can then rightfully ask: can we infer the functions $g$ and $\lambda$?

This question naturally fits a non-parametric  \emph{Bayesian} framework in which the functions $g$ and $\lambda$ are assumed to be unknown, and to be distributed according to a \textit{prior} measure over functions. The explicit formulas developed in this paper would be nothing but the likelihood of the observations given the unknown functions $g$, $\lambda$.  With these formulas at hand, a Markov chain Monte Carlo algorithm could be used to estimate expectations under the \textit{posterior} distribution of unknowns given observations. In this way we could make predictions about where the location of the next fastest server will be and quantify how confident we are about our prediction.

\subsubsection{Dependence of rates and delays on density }
Suppose the points $x_1, \dots, x_n$ are distributed according to the density $\rho(x) dx$. There may be applications where the rate function $\lambda$ and the latency function $g$ depend on $x$ only through the density $\rho$ of points around $x$ . In particular, we may consider a model of the form
\[ \lambda(x) := \Lambda (\rho(x) ), \quad g(x):= G(\rho(x)), \]
where $\Lambda$ and $G$ are scalar functions. For example, we may imagine a parametric model of the form
\[ \Lambda(t) = t^{\alpha} , \quad G(t) = t^{\beta}, \]
where $\alpha$ and $\beta$ are real numbers. Estimating the functions $\lambda$ and $g$ then reduces to learning the parameters $\alpha$ and $\beta$. Estimating these parameters may provide valuable qualitative, as well as quantitative, information about the structure of latency and rate patterns due to high or low concentration of servers.

\subsubsection{Extension to weakly-correlated response times}
The response times  $f_n(x_1), \dots, f_n(x_n)$ \nc in Theorem \ref{thm:main2} were assumed to be independent. Nevertheless, it is possible to extend our results to \textit{weakly-correlated} processing times, allowing spatially correlated perturbations. Precisely, one could consider
=
\[  f_n \nc (x_i) = \frac{n}{\lambda(x_i)}\xi_i + g(x_i)   \]
\nc
where $g$ is a random field independent from $\{ \xi_1, \dots, \xi_n\}$ for which, with probability one, 
\[ || g||_\infty < \infty, \]
and where the $\xi_i$ are independent. In this context, to obtain the asymptotic distribution of the first $k$ argmins $x^{(1)}, \dots,x^{(k)}$ of $f_n$ \nc we may use the independence lemma to first obtain the asymptotic distribution of $(x^{(1)}, \dots, x^{(k)})| g $ using Corollary \ref{cor:main} and then integrate with respect to the distribution of $g$. The latency function $g$ could be modeled according to a Gaussian random field.

\subsubsection{Extensions to network structures}
In the setting of server responses, it may be more accurate to consider a graph structure as opposed to an Euclidean one. Thus, consider a very large graph $G_n$ with vertex set $V_n$. Suppose that to each node $v$ in the graph we associate a server which can finish a task in time $f_v$. It is natural to ask whether one can learn the latency, rate, and weak correlation structure of processing times in such a large graph.

A possible practical approach to answer this question is to use an embedding of the set of nodes $V_n$ into Euclidean space. For example, we can consider a spectral map $\Theta : V_n \rightarrow \R^d$ constructed using the graph Laplacian  and then imagine that the servers are actually located at the ``geographic'' locations $\Theta(v_1), \dots, \Theta(v_n)$. One could then use the Bayesian inference methods described above to estimate the desired quantities. The idea of using the map $\Theta$ is that points $v, \tilde{v}$ are ``close'' when the points $\Theta(v), \Theta(\tilde{v})$ are close to each other; in this way we translate the non-geographic information contained in $G_n$ into geographic information which fits the set-up considered in this paper. Specific applications of these ideas are to be explored in the future.

\subsection{Outline}\label{ssec:outline}

The remainder of the paper is organized as follows. In Section \ref{sec:prelim} we collect some background results and we establish some properties of the process $W_\lambda.$   The main results are proved in Section \ref{sec:main-proofs}. We conclude in Section \ref{sec:comp-ex} with a short illustrative example.

\section{Preliminaries}\label{sec:prelim}
This section contains background results that will be employed in the proof of our main theorems. Subsection \ref{sec:topology} describes the function space $\mathcal{S}(\overline{D})$ of lower semicontinuous functions endowed with a suitable topology.  Subsection \ref{sec:W-props} establishes some properties and constructions of the process $W_\lambda$.

\subsection{The space $\S(\overline{D})$ and the  topology of $\Gamma$-convergence}\label{sec:topology}
Our goal here is to introduce the topology of $\Gamma$-convergence on the space of lower semi-continuous functions. This topology generates a notion of convergence which preserves the structure of minima and minimizers. $\Gamma$-convergence is also known as epi-convergence in the probability literature \cite{VervaatTopology}. This topology has found application in many fields, e.g. materials science, Ginzburg-Landau theory, and image processing. We will follow the presentation in \cite{DalMaso}, Chapter 10.

For any metric space $X$ we define $\S(X)$ to be the family of lower semi-continuous functions on $X$ taking values in $\R\cup \{\infty\}$. The following definition is standard:

\begin{defn}
  Let $X$ be a metric space. A sequence $\{f_n\}_{n=1}^\infty \subset \S(X)$ is said to $\Gamma$-converge to $f \in \S(X)$ (written $f_n \gto f$) if

  \begin{enumerate}
    \item For all $x \in X$ there exists a sequence satisfying $x_n \to x$ such that
      \begin{equation}
	f(x) \geq \limsup_{n \to \infty} f_n (x_n).
	\label{eqn:g-conv-def}
      \end{equation}
    \item For all $x_n \to x$ we have that
      \begin{equation}
	f(x) \leq \liminf_{n \to \infty} f_n(x_n).
	\label{eqn:g-conv-def-2}
      \end{equation}
  \end{enumerate} 
\end{defn}

 In essence this definition requires that the limiting object $f$ takes values below any possible limit, but that the value of $f$ is obtained by some appropriately chosen ``recovery sequence''. This notion of convergence imposes essentially the weakest conditions needed to guarantee convergence of minima and minimizers. This is manifest in the following proposition which is of high relevance for our purposes:

\begin{prop}
	\label{prop:mainGamma}
  Suppose that $f_n \gto f$, in some metric space $X$. Then $\min f_n \to \min f$. If $X$ is also compact then if $x_n^* \in \argmin f_n$ then any limit point of $x_n^*$ will be a minimizer of $f$. 
\end{prop}

It is natural to seek a topology which describes $\Gamma$-convergence. To this end, we define the following topologies:

\begin{defn}
  \begin{enumerate}
    \item We let $\sigma^+$ be the topology generated by sets of the form $\left\{f \in \S(X): \inf\limits_U f < t\right\}$, where $U$ is any open subset of $X$ and $t \in \R$.
    \item We let $\sigma^-$ be the topology generated by sets of the form $\left\{f \in \S(X): \inf\limits_K f > t\right\}$, where $K$ is any compact subset of $X$ and $t \in \R$. 
    \item We let $\sigma$ be the smallest topology containing $\sigma^+$ and $\sigma^-$.
  \end{enumerate}
\end{defn}

These topologies permit the measurement of minima on open and closed sets. It turns out that this topology is equivalent to $\Gamma$-convergence in the following sense (c.f. Theorem 10.17 in \cite{DalMaso})
\begin{prop}
  Let $X$ be a metric space. A sequence  $\{f_n\}_{n=1}^\infty \subset \S(X)$ converges in $\sigma$ if and only if it $\Gamma$-converges. 
\end{prop}

This space of functions is somewhat different from many of the standard function spaces. For example, it possesses the following compactness property (c.f. Theorem 10.6 in \cite{DalMaso}):

\begin{thm}
  The topological space $(\S(X),\sigma)$ is a compact space.
  \label{thm:tau-space-compact}
\end{thm}

An immediate application, which is of importance to the present investigation, is the following:

\begin{cor}
  Any sequence $\{f_n\} \in \S(X)$ has a subsequence which $\Gamma$-converges.
\end{cor}

In general the topology generated by $\Gamma$-convergence will not be metrizable (or even Hausdorff). However, in the setting where $X$ is compact we have the following as a consequence of Corollary 10.23 in \cite{DalMaso}; see also Theorem 5.5 in \cite{VervaatTopology}:

\begin{prop}\label{Prop:metrizable}
  Let $X$ be a \emph{compact} metric space. Then $\sigma$ is metrizable.
\end{prop}


Finally, in this paper we will work with $X= \overline{D}$ where $D$ is a bounded domain in $\R^d$. This allows us to characterize the topology in terms of cubes. The following is a direct application of Theorem 5.3 in \cite{VervaatTopology}:

\begin{prop}\label{prop:sigma-alg}
  For $X := \overline{D}$ where $D \subset \R^d$ is an open bounded domain,  $\sigma^\pm$ are generated by evaluating the infimum of functions over $U$ being open cubes and $K$ being closed cubes (as opposed to arbitrary open and closed sets).

\end{prop}


\subsection{The process $W_\lambda$}\label{sec:W-props}

In this subsection we establish some properties of the extreme value process $W_\lambda$ introduced in Definition \ref{def:W}. We first show the existence and uniqueness of $W_\lambda$, and then the existence an uniqueness of its $k$-th argmins. We will denote by $\mathcal{P}(\S)$  the space of probability measures over the space of lower semi-continuous functions on $\bar{D}$.

\subsubsection{Existence and uniqueness of the process $W_\lambda$}

We first show that the distribution of $W_\lambda$ is a well-defined object. 

\begin{prop}\label{prop:W-unique}
  Any two random variables with values in $\mathcal{S}(\overline{D})$ satisfying Definition \ref{def:W} have the same distribution. 
\end{prop}

\begin{proof}
  Suppose that $\Prob_1,  \Prob_2 \in \mathcal{P}(\S)$ are the distributions of two random variables satisfying Definition \ref{def:W}. Let
  \begin{displaymath}
    \mathcal{A} = \{ E \subset \mathcal{S}(D) | E \in \sigma, \Prob_1(E) = \Prob_2(E)\}.
  \end{displaymath}
  It is straightforward to verify that $\mathcal{A}$ is a $\lambda$-system.

  Next, by Property 1 in Definition \ref{def:W}, we have that any set of the form $\{ \inf_K f > s\}$ is in $\mathcal{A}$ for any closed cube $K$ and real number $s$. Similarly, by taking a limit of cubes from the outside and again using Property 1 in Definition \ref{def:W} we have that any set of the form $\{ \int_U f < t\}$ is in $\mathcal{A}$ for any open cube $U$ and real number $t$. By using Property 2 in Definition \ref{def:W} we will have that any finite intersection of sets of these forms will also be in $\mathcal{A}$. Thus $\mathcal{A}$ contains the $\pi$-system generated by sets of these forms. By the $\pi-\lambda$ theorem, $\mathcal{A}$ contains the sigma algebra generated by sets of these forms. By Proposition \ref{prop:sigma-alg} it follows that $\mathcal{A} = \sigma$. This implies that $\Prob_1 = \Prob_2$, which concludes the proof.
\end{proof}

Now we will demonstrate the existence of $W_\lambda$ by constructing an appropriate approximating sequence.

\begin{prop}
\label{ConvergenceGamma}
Let $\{x_1, x_2, \dots \}$ be a countable subset of $\overline{D}$ and for every $n \in \N$ let 
\[  \mu_n:= \frac{1}{n}\sum_{i=1}^n \delta_{x_i}. \]
We assume that $\mu_n$ converges weakly towards the distribution with density proportional to $\lambda(x)$. 
Let $\{\xi_1, \xi_2, \dots \}$ be i.i.d. exponentially-distributed random variables with rate one. Then, the random functions
\[ W_n(x) :=  \begin{cases}    \frac{n}{  \overline{\lambda}}  \xi_i   , \quad  \text{ if } x=x_i, \\ + \infty,  \quad \text{ otherwise,}        \end{cases}\]
converge weakly (in $\mathcal{P}(\S)$) towards $W_\lambda$, where in the formula for $W_n$ we are using
\[ \overline{\lambda}:= \int_D  \lambda(x) dx.  \]
\end{prop}
\begin{proof}

  Since the space $(\mathcal{S}(\overline{D}),\sigma)$ is compact, any sequence in $\mathcal{P}(\mathcal{S})$ is tight. Therefore, using that $(\mathcal{S}(\overline{D}),\sigma)$ is a separable, complete, metrizable space, Prokhorov's theorem implies that $W_n$ must have a limit point (in the sense of weak convergence of measures) in $\mathcal{P}(\mathcal{S})$.

Let $\tilde W$ be some limit point of the $W_n$. Our goal is to show that $\tilde W$ satisfies the conditions given in the definition for $W_\lambda$. To that end, first consider a closed set $C$. From the definition of $W_n$, at every point $x_i$, $i=1\dots n$ we have an independent exponential variable, with rate $\frac{\bar \lambda}{n}$. This then implies that $\min_{C \cap \overline D} W_n$ is exponentially distributed with rate $\mu_n(C \cap \overline D) \bar \lambda$. As $\mu_n \rightharpoonup \frac{\lambda}{\overline \lambda} dx$, we have that $$\lim_{n \to \infty}\min_{C \cap \overline D} W_n$$ is exponentially distributed with rate $\int_{C \cap \overline D} \lambda(x) \,dx$. Since taking mins over closed sets is measurable in $\mathcal{P}(\S)$, $\tilde W$ must satisfy the first point in the definition of $W_\lambda$.

For the second property, we notice that the min over a finite number of disjoint closed sets will be independent under $W_n$, as the $\xi_i$ are all independent. These events will also be in the topology $\sigma$. Hence by taking limits we obtain that $\tilde W$ will satisfy the second property in the definition of $W_\lambda$. The uniqueness of $W_\lambda$ given by Proposition \ref{prop:W-unique} then gives that the unique limit point of $W_n$ in $\mathcal{P}(\S)$ is precisely $W_\lambda$, which completes the proof.

\end{proof}

The previous construction of $W_\lambda$ closely mirrors the framework described in the introduction: namely that we trace the rescaled minimum of many exponentially-distributed variables. An alternative construction (which can be connected to the Poisson process construction of extremal processes in \cite{Pickands71}) is also possible. This construction makes certain properties of the $W_\lambda$ simpler to visualize, and we include it for completeness.

\begin{prop} \label{prop:ConvergenceGamma-alt}
  Let $\{x_1,x_2, \dots\}$ be a sequence of points, i.i.d according to the probability measure $\frac{\lambda}{\overline \lambda} dx$ . Let $\{\xi_1,\xi_2,\dots\}$ be an i.i.d. sequence of random variables distributed as $\exp(\bar \lambda)$. Define the following random functions:
  \begin{displaymath}
    W_n(x) = \begin{cases} \sum\limits_{j=1}^i \xi_j,\quad &\text{ for } x = x_i, \enspace i \leq n, \\ + \infty, &\text{ otherwise}.\end{cases}
  \end{displaymath}
  Then the $W_n$ converge weakly (in $\mathcal{P}(\mathcal{S})$) towards $W_\lambda$.
\end{prop}

\begin{proof}
  As in the previous construction, we only need to prove that $W_n$ asymptotically satisfies Properties 1 and 2 from Definition \ref{def:W}. We will do this by direct computation.
  
  For Property 1 in Definition \ref{def:W}, given any closed set $C$, the probability that $x_k$ lies in $C$ is equal to $\bar \lambda^{-1} \int_{C \cap \overline D} \lambda(x)dx =: \tilde \lambda_C$. Thus the first $i$ such that $x_i \in C$ is geometrically distributed with parameter $\tilde \lambda_C$. Furthermore,  $\sum_{j=1}^i \xi_j$ has an Erlang distribution with density $\frac{\bar \lambda^i t^{i-1} e^{-\bar \lambda t}}{(i-1)!}$. In turn
  \begin{displaymath}
    \mathbb{P}\left(\min_{x \in C \cap \overline D} W_n \geq r\right) = \int_r^\infty\sum_{j=1}^n (1-\tilde \lambda_C)^{j-1} \tilde \lambda_C\frac{\bar \lambda^j s^{j-1} e^{-\bar \lambda s}}{(j-1)!} ds + \sum_{j=n+1}^{\infty}(1- \tilde \lambda_{C})^{j-1} \tilde{\lambda}_C,
  \end{displaymath}
  where on the right we are using the fact that the choice of points $x_i$ is independent of the values of the $\xi_j$. Taking a limit as $n \to \infty$ and simplifying the series, we find that
  \begin{displaymath}
    \lim_{n \to \infty} \mathbb{P}\left(\min_{x \in C \cap \overline D} W_n \geq r \right) = \int_r^\infty \tilde \lambda_C \bar \lambda \exp(-\tilde \lambda_C \bar \lambda s) ds.
\end{displaymath}
This proves Property 1.

With regards to Property 2, for $r < s$ we can compute (letting $p_n$ be the associated density function):
\begin{align*}
  p_n(\min_{C_1} W_n &= r, \min_{C_2} W_n = s) = \sum_{k_1=1}^n (1-\tilde \lambda_{C_1} - \tilde \lambda_{C_2})^{k_1-1} \tilde \lambda_{C_1} \frac{\bar \lambda^{k_1}r^{k_1-1}e^{-\bar \lambda r}}{(k_1-1)!} \\
  &\times \sum_{k_2 = k_1+1}^n (1-\tilde \lambda_{C_2})^{k_2-k_1-1} \tilde \lambda_{C_2} \frac{\tilde \lambda^{k_2-k_1}(s-r)^{k_2-k_1-1} e^{-\bar \lambda (s-r)}}{(k_2-k_1-1)!}\\
  &= \sum_{k_1=1}^n (1-\tilde \lambda_{C_1} - \tilde \lambda_{C_2})^{k_1-1} \tilde \lambda_{C_1} \frac{\bar \lambda^{k_1}r^{k_1-1}e^{-\bar \lambda r}}{(k_1-1)!} \\
  &\times \sum_{\tilde k = 1}^{n-k_1} (1-\tilde \lambda_{C_2})^{\tilde k-1} \tilde \lambda_{C_2} \frac{\bar \lambda^{\tilde k}(s-r)^{\tilde k-1} e^{-\bar \lambda (s-r)}}{(\tilde k-1)!}.
\end{align*}
Taking $n \to \infty$ we then obtain
\begin{align*}
  &\lim_{n \to \infty} p_n(\min_{C_1} W_n = r, \min_{C_2} W_n = s) \\
  &= \tilde \lambda_{C_1} \bar \lambda \exp \left( -\bar \lambda r + \bar \lambda r(1-\tilde \lambda_{C_1} - \tilde \lambda_{C_2})\right) \times \tilde \lambda_{C_2} \bar \lambda \exp \left( -\bar \lambda (s-r) + (s-r)\bar \lambda(1-\tilde \lambda_{C_2}) \right)\\
  &= \tilde \lambda_{C_1} \bar \lambda \exp (-\bar \lambda \tilde \lambda_{C_1} r) \times \tilde \lambda_{C_2} \bar \lambda \exp (-\bar \lambda \tilde \lambda_{C_2} s).
\end{align*}
This proves Property 2 in the case of two sets. The case with more than two sets is completely analogous.
\end{proof}

\begin{rem}
  The previous construction highlights the memorylessness of the process $W_\lambda$. That is, the distribution of the $k$-th min can be found by restarting the process after the arrival of the $k-1$-th min. This memorylessness will be very convenient in establishing integral formulas in Section 
  \ref{sec:main-proofs}. However, one cannot expect this property to hold for other extreme value distributions, see Remark \ref{rem:alpha-stable-analog}.
\end{rem}

\begin{rem}
  There are several classical constructions of Brownian motion. Some, such as L\'evy's piecewise linear construction, are easy to visualize but only converge as a measure on the space of continuous functions. Others, such as the ``wavelet-type'' construction of L\'evy-Ciesielski, or the Fourier construction of Weiner, converge uniformly towards Brownian motion (see e.g. Chapter 3 in \cite{Schilling-Brownian-Motion} for more details). Here we only characterize the convergence of the distribution of $W_n$ towards $W_\lambda$ as measures on the space of lower semi-continuous functions. One could seek to demonstrate that certain alternative constructions converges uniformly, in an appropriate metric on the space of lower semi-continuous functions as in Proposition \ref{Prop:metrizable} (see also \cite{DalMaso}, Proposition 10.21). This will be the subject of future analysis. 
\end{rem}

\subsubsection{Existence and uniqueness of $k$-th argmins of $W_\lambda$}

Next we turn to studying the existence of $k$-th argmins for the process $W_\lambda$. We begin with a definition:
\begin{defn}\label{def:k-mins}
  Let $f \in \mathcal{S}(\overline{D})$. We define the minimum value of $f$
\[m_1(f) := \min_{x \in \D} f(x), \]
as well as the set of minimizers of $f$
\[ M_1(f):= \{  x \in \D \: : \: f(x) = m_1(f)  \}.  \]
Moreover, having defined the numbers $m_1(f), \dots, m_{k-1}(f)$ and the sets $M_1(f), \dots, M_{k-1}(f)$ we define
\begin{equation}\label{eqn:def-m_k}  m_k(f) := \inf_{ x \in \D \setminus \cup_{i=1}^{k-1} M_i(f) } f(x)  \end{equation}
and
\[ M_k(f) := \{  x \in \D \: : \:  m_{k-1}(f) <  f(x)= m_{k}(f)    \}.  \]
We refer to the elements of $M_k(f)$ as $k$-th argmins of $f$. 
\end{defn}

\begin{rem}
The set $M_1(f)$ is always non-empty due to the compactness of $\D$ and the lower semi-continuity of $f$. The lower semi-continuity of $f$ also implies that for every $k \in \N$, the set
\[ \bigcup_{i=1}^{k} M_i(f) \]
is a closed set.
\end{rem}

\begin{rem}
In general the sets $M_k(f)$ for $k>1$ may be empty.    For example, it should be clear that a continuous function $f$ does not have $2$-th argmins. In this paper, however, the random object that we consider is almost surely a lower semi-continuous function with well-defined (and unique) $k$-argmins for all $k$ (see Proposition \ref{prop:kmins} below).
\end{rem}

\begin{remark}
It is an easy exercise to show that if $M_k(f)$ is a non-empty set then $M_{k-1}(f) $ is also non-empty.
\end{remark}

The following lemma will be critical in proving the existence of $k$-argmins of $W_\lambda+g$.

\begin{lem}
	\label{lem:kth}
	Let $\{f_n\} \subset \mathcal{S}(\overline{D})$, let $k \in \N$ and suppose that:
	\begin{enumerate}
		\item $f_n \gto f \in \mathcal{S}$.
		\item The set $M_i(f)$ is a singleton for all $i\leq k$.
		\item There exists a number $\beta>0$ such that
		\begin{equation}
		  d(M_i(f_n), M_{j}(f_n)) > \beta, \quad  \forall n \in \N  , \quad \forall  i,j=1, \dots, k+1, \quad i \not =j.
		  \label{eqn:beta}
		\end{equation}  
		In the above, $d(M_i(f_n), M_{j}(f_n))$ is defined as
		\[   d(M_i(f_n), M_{j}(f_n)):= \inf \{ | x-y | \: : \: x\in M_i(f_n), \quad y \in M_{j}(f_n) \}. \]
		
		\item For $n$ sufficiently large $M_i(f_n) \neq \emptyset$, for all $i \leq k+1$.
	\end{enumerate}
	Then all limit points of $\{x_{n}^{k+1}\}$, where $x_{n}^{k+1} \in M_{k+1}(f_n)$, belong to $M_{k+1}(f)$. In particular, $M_{k+1}(f)$ is non-empty.
	\label{lem:k-min-lim}
\end{lem}

\begin{proof}
	
	\textbf{Step 1:} We start by proving the case $k=1$ in order to illustrate the ideas.
	
	Let  $x^1\in M_1(f)$ be the unique minimizer of $f$. By compactness of $\overline{D}$ we know that $\{ x_n^{2} \}_{n \in \N}$ converges up to subsequence towards a point $\tilde{x}^2 \in \overline{D}$; without  loss of generality we assume that the full sequence converges towards $\tilde{x}^2$. We need to show that $\tilde{x}^2\in M_2(f)$. 
	
	First, we claim that $\tilde{x}^2 \not = x^1$, and so $\tilde{x}^2 \notin M_1(f)$. To see this, let $x_n^1$ be a minimizer for $f_n$. By the $\Gamma$-convergence assumption and the fact that $x^1$ is the unique minimizer of $f$, it follows that $ x_n^1\rightarrow x^1.$  
	In particular, for large enough $n$ we have
	\[  |   x_n^1 - x^1| < \beta/2, \]  
	where $\beta$ is as in \eqref{eqn:beta}. 
	From the triangle inequality it follows that for large enough $n$ 
	\[  | x_n^{2} - x^1 | \geq |x_n^1 - x_n^{2} |  - | x_n^1 - x^1|  \geq \beta - \beta/2  = \beta/2.  \]
	Therefore,
	\[ | \tilde{x}^2 - x^1 | \geq \beta/2 >0,\]
	 establishing the claim.
	
	Next, we show that for arbitrary $x \in \overline{D} \setminus M_1(f)   =\overline{D} \setminus \{ x^1\}$ we have that $f(\tilde{x}^2) \le f(x).$ This implies that $\tilde{x}^2 \in M_1(f) \cup M_2(f)$, which together with the above shows that  $\tilde{x}^2 \in M_2(f).$
	Using the lim-sup inequality we can find a sequence $\{ x_n \}_{n \in \N}$ converging towards $x$ for which
	\[ \limsup_{n \rightarrow \infty}  f_n(x_n) \leq f(x). \]
	Notice that for all large enough $n$, $x_n \not = x_n^1$. Indeed, if that was not the case, the limit of $\{ x_n \}_{n \in \N}$ would be $x^1$ which contradicts that $x\neq x^1$. In particular, the value of $f_n(x_n^2) \le f_n(x_n)$.  Hence, 
	\[  f(\tilde{x}^2) \leq \liminf_{n \rightarrow \infty} f_n(x_n^{2}) \leq \limsup_{n \rightarrow \infty} f_n(x_n) \leq f(x)  ,   \]
	where the first inequality follows from the liminf inequality given that $x_n^{2} \rightarrow \tilde{x}^2$.

\textbf{Step 2: }Generalizing to arbitrary $k$ follows the same ideas.

Let, for $i = 1,\dots, k$,  ${x^i}\in M_i(f)$ be the unique $i$-th argmin of $f$ and let, for $i = 1, \dots, k+1,$  ${x_n^i}\in M_i(f_n).$ We assume without loss of generality that ${x_n^i} \to \tilde x^i$ for some $\tilde x^i \in \overline D$. Our goal is to prove that $\tilde x^{k+1}\in M_{k+1}(f)$. 
	
Restricted to the set $B(\tilde x^i, \beta/2)$ we have that  for $n$ sufficiently large ${x_n^i}$ minimizes $f_n$. $\Gamma$-convergence then implies that $f_n({x_n^i}) \to f(\tilde x^i)$. By the definition of $M_i(f_n)$ we have that, for $i<j$, $f_n({x_n^i}) < f_n({x_n^j})$ and hence $f(\tilde x^i) \leq f(\tilde x^j)$. Note that $\tilde x^1 = {x^1}$, simply by $\Gamma$-convergence (on the whole space $\overline D$) and the fact that $M_1(f)$ is a singleton.
	
Now, we claim that for all $i = 1,\dots, k$ we have that $\tilde x^i = x^i$. We proceed by induction. The base case was already proved. Suppose that $\tilde x^i = {x^i}$ for all $i = 1,\dots, j< k$. If $\tilde x^{j+1} \neq {x^{j+1}}$ then we must have (by the fact that $M_i(f)$ is a singleton) that $f(\tilde x^j) < f({x^{j+1}}) < f(\tilde x^{j+1})$. By $\Gamma$-convergence there exists a sequence $x_n \to {x^{j+1}}$ so that $f_n(x_n) \to f({x^{j+1}})$. This then implies that $f_n({x_n^j}) < f_n(x_n) < f_n({x_n^{j+1}})$ for $n$ sufficiently large. This violates the fact that ${x_n^{j+1}} \in M_{j+1}(f_n)$, which then proves the claim.
	
Finally, let $x \in \overline D \backslash \cup_{i=1}^k M_i(f)$ be arbitrary. We will show that $f(\tilde x^{k+1}) \le f(x)$ which implies that $\tilde{x}^{k+1} \in M_{k+1}(f).$  By $\Gamma$-convergence, there exists a sequence $x_n \to x$ with $f_n(x_n) \to f(x)$. Since $x \notin \cup_{i=1}^k M_i(f)$ we have that, for $i = 1,\dots, k,$ $f(x) > f({x^i})$  and so, for $n$ large enough, $f_n(x_n) > f_n({x_n^i})$. Because ${x_n^{k+1}} \in M_{k+1}(f_n)$, this in turn implies that $f_n(x_n) \geq f_n({x_n^{k+1}})$. Taking limits concludes the proof. 
\end{proof}

We are now ready to show the existence and uniqueness of $k$-argmins for $W_\lambda + g$.

\begin{prop}\label{prop:kmins}

Suppose that $g: \overline{D} \rightarrow \R$ is a continuous function. Then with probability one, for every $k \in \N$, $M_k(W_\lambda+g)$ is a singleton. In other words $W_\lambda+g$ has well-defined and unique $k$-th argmins for every $k \in \N$.  
\end{prop}

\begin{proof}
By Proposition \ref{prop:ConvergenceGamma-alt} we can assume without loss of generality that, with probability one, $W_n$ as defined in \eqref{prop:ConvergenceGamma-alt} $\Gamma$-converges towards $W_\lambda$ . Since $g$ is continuous it follows that, with probability one, 
\[ W_{n}+ g \gto W_\lambda+ g. \]

From the construction of the $W_n$ it is easy to see that, with probability one, the sequence of functions $\{ W_n + g \}_{n \in \N}$ satisfies conditions 1-3 in Lemma \ref{lem:kth}. In order to show the existence of $k$-argmins it is enough to show that with probability one $M_k(W_\lambda+ g)$ cannot have more than two elements (proving in this way existence and uniqueness).


We only consider the case $k=1$ as the general case is proved similarly.  Let  $\mathcal{F}_l$ be the family of dyadic cubes $Q$ with side length $2^{-l}$. Suppose that $M_1(W_\lambda + g)$ has at least two elements. Then for some $l \in \N$ we have that $\min_Q W_\lambda + g = \min_{Q'} W_\lambda + g$ for two disjoint $Q,Q' \in \mathcal{F}_l$. However, by Properties 1 and 2 of Definition \ref{def:W} we have that
  \[ \mathbb{P}\left(\bigcup_{l=1}^{\infty}\; \bigcup_{\substack{Q, Q' \in \mathcal{F}_l \\ Q\cap Q' = \emptyset}}  \left\{ \omega \in \Omega \: : \: \min_{x \in Q} W+g = \min_{x \in Q'} W+g  \right\}\right) = 0.\]
This then implies that $M_1(W_\lambda + g)$ must be a singleton (with probability 1).

%

 \end{proof}

 \begin{rem}
   We have used the fact that argmins are locally isolated in a very strong way. Ideas of this type have previously been applied in the context of Cahn-Hilliard phase transitions, see e.g. \cite{KohnSternberg}.
 \end{rem}

 \section{Proof of the main results}\label{sec:main-proofs}
This section contains the proofs of the main results of the paper, Theorems  \ref{thm:main} and \ref{thm:main2}

\subsection{Proof of Theorem \ref{thm:main}}
We already showed in Proposition \ref{prop:kmins} that $W_\lambda + g$ has well-defined $k$-th argmins. Our aim now is to establish equation \ref{eqn:argmin-dist}, that gives a formula for their joint distribution. We will consider first the case $k=1$ in Lemma \ref{thm:main1}, and then extend the proof to the case $k> 1.$

Before delving into the proofs we recall some notation ---see equation \eqref{eqn:H-def}. We denote by $\mu_{\lambda}$ the probability measure on $\overline{D}$ with density proportional to $\lambda$, 
\[ \frac{d\mu_\lambda(x)}{dx} \propto \lambda(x) , \quad x \in \overline{D}.\]
The normalizing constant is denoted by $\bar{\lambda},$ that is, 
\[ \overline{\lambda}:= \int_{D}\lambda(x) dx.\]
We also recall the definition of the cumulative distribution function
\[ H_{\lambda,g}(s) := \mu_{\lambda} \left(  \{  x \in D \: : \:  g(x) \leq s   \}   \right). \]
Henceforth we will omit the dependence of $\mu_\lambda$  and $H_{\lambda,g}$ on $\lambda$ and $g$ from the notation, and write simply $\mu$ and $H.$

%

      \begin{lem}  \label{thm:main1}

    Let $g \in \mathcal{S}(\overline{D})$, and let $\Xone$, $\tauone$ be random variables representing the global minimizer and minimum value of $W_\lambda + g$. Then the joint density of these random variables is given by
    \begin{equation}
      \frac{\Prob(\Xone \in dx,\tauone\in d\tau )}{dx d\tau} = \Psi(x,\tau,g) = \chi_{\tau > g(x)} \exp\left(-\bar \lambda \int_{-\infty}^{\tau}H(s) \,ds \right).
      \label{eqn:joint-distribution-1-min}
    \end{equation}
    In turn the density of $\Xone$ is given by
    \begin{equation}
\frac{\Prob(\Xone \in dx) }{dx} = \Phi(x,g) = \int_{g(x)}^\infty  \exp \left( -\bar \lambda \int_{-\infty}^{t}H(s) \,ds \right) \,dt.
      \label{eqn:1-min-dist}
    \end{equation}

    \end{lem}
    \nc
    \begin{proof}
Let $\{ x_1, x_2, \dots \} $ be a countable set of points in $D$ that we use to construct discrete approximations of $\mu$ and $H$. Precisely, we set, for every $n \in \N,$  $$\mu_n:= \frac{1}{n}\sum_{i=1}^{n} \delta_{x_i},  \quad \quad H_n(\alpha) := \mu_n \left( \{ x \in D \: : \: g(x)\leq \alpha\} \right). $$
The points $\{x_1, x_2, \dots \}$ can be chosen so that $\mu_n$ converges weakly towards $\mu$, and $H_n$ converges uniformly towards $H$.

Let $\{\xi_1, \xi_2, \dots \}$ be a sequence of i.i.d. exponential random variables with rate one. Fix $n \in \N$ and consider the (random) function

\[  \bar{f}_n(x) := \begin{cases}   \frac{n}{ \overline{\lambda}} \xi_i + g(x_i), & \text{ if } x=x_i, \quad i=1, \dots, n, \\ +\infty, & \text{otherwise. }       \end{cases} \]
Let $\Xone_n$ and $\tauone_n$ denote the minimizer and minimum of the function $\bar{f}_n$. We may compute, for $r > 0$,

\begin{align*}
&\Prob\left(\Xone_n  = x_j, \tauone_n = r + g(x_j)\right) = \Prob\left(  \frac{n}{\bar{\lambda}} \xi_j = r, \frac{n}{\bar{\lambda}}\,\xi_j \leq \frac{n}{\bar{\lambda}}\xi_i + (g(x_i) - g(x_j))\; \forall i\right) \\
&= \Prob\left(\xi_j = r\frac{\bar{\lambda}}{n},\,\, \xi_i>\frac{\bar{\lambda}}{n}\bigl(r + g(x_j) - g(x_i)  \bigr) \right) \\
&= \exp\left( -r \frac{\bar{\lambda}}{n} \right)  \int_{\frac{\bar{\lambda}}{n} (r - (g(x_1)  - g(x_j)) )_+}  \cdots \int_{\frac{\bar{\lambda}}{n} (r- (g(x_n)  - g(x_j)) )_+}  \prod_{i \neq j} \exp(  - r_i )    \,dr_i \\
& =\exp\left( -r \frac{\bar{\lambda}}{n} \right)  \exp\left(\frac{\bar{\lambda}}{n} \sum_{g(x_i) < r + g(x_j)}^{}\bigl(g(x_i) - g(x_j) - r \bigr) \right).
\end{align*}
Using the layer cake decomposition we may write
 \begin{displaymath}
   \Prob(\Xone_n = x_j, \tauone_n = r + g(x_j)) = \exp\left( -r \frac{\bar{\lambda}}{n} \right) \exp\left( -\bar \lambda \int_{-\infty}^{r + g(x_j)} H_n(s) \,ds \right).
 \end{displaymath}
 Replacing $r + g(x_j)$ with $\tauone_n$ and letting $n \to \infty$ we obtain
 \begin{equation}\label{eqn:x-t-density-b}
   \frac{\Prob(\Xone \in dx , \tauone \in d \tau)}{dx d\tau} = \Psi(x,\tau,g) := \chi_{\tau > g(x)} \exp\left( -\bar \lambda \int_{-\infty}^{\tau} H(s) \,ds \right).
 \end{equation}
Now taking the integral in $\tau$ gives
 \begin{displaymath}
   \frac{\Prob(\Xone \in dx )}{dx} = \Phi(x,g) := \int_{g(x)}^\infty  \exp\left( -\bar \lambda \int_{-\infty}^{t} H(s) \,ds \right)\,dt.
 \end{displaymath}
Proposition \ref{ConvergenceGamma} implies that, with probability one, $f_n \gto W_\lambda + g$. This in turn implies the desired result thanks to Proposition \ref{prop:mainGamma}.

    \end{proof}
    We are now ready to establish the formulas for the joint distribution of the first $k$-th argmins.

\begin{proof}[Proof of Theorem \ref{thm:main}:]
We complete the proof of Theorem \ref{thm:main} by obtaining the joint distribution of $(\Xone, \dots, \Xk)$. 

We first observe that
\begin{align*}
  \Prob(\Xii \in A | \Xone , \ldots \Ximone, \tauone, \ldots,\tauimone) &= \Prob(\Xii \in A | \tauimone),\\
  \Prob(\Xii \in A, \taui \in B| \Xone, \ldots, \Ximone, \tauone,\ldots, \tauimone) &= \Prob(\Xii \in A, \taui \in B| \tauimone).
\end{align*}

Bayes' rule then implies that
\begin{equation}
  \frac{\Prob(\Xone \in dx_1, \dots, \Xk \in dx_k) }{dx_1\dots dx_k} = \int_{\R^{k-1}} \frac{\Prob(\Xk | \taukmone)}{dx_k} \frac{\Prob(\Xone,\tauone)}{dx_1 dr_1} \prod_{j=2}^{k-1} \frac{\Prob(\Xj,\tauj|\taujmone)}{dx_j dr_j} \,dr_1\dots\,dr_{k-1}.
\end{equation}
Here we use $r_i$ as a dummy variable to avoid overloading notation for $\tau$. Using the fact that exponentials are memoryless (along with the construction in Proposition \ref{ConvergenceGamma} and the convergence in the previous proof) we then obtain
\begin{align}
  \frac{\Prob(\Xone \in dx_1, \dots, \Xk \in dx_k )}{dx_1 \dots dx_k} = &\int_{\R^{k-1}} \Phi( x_k,(g-r_{k-1})^+) \Psi(x_1,r_1,g) \\
  &\times \prod_{j=2}^{k-1} \Psi(x_j,r_j-r_{j-1},(g-r_{j-1})^+) \,dr_1\dots\,dr_{k-1}.
\end{align}
This matches the formula in Theorem \ref{thm:main} and completes the proof.

\end{proof}

\subsection{ Proof of Theorem \ref{thm:main2}}
\begin{proof}[Proof of Theorem \ref{thm:main2}]
To begin with, we provide the (classical) proof that
\begin{displaymath}
  n \min\{ \xi_1 \dots \xi_n \} \to^d \exp(1).
\end{displaymath}
Define $Y_n := n \min\{\xi_1 \dots \xi_n\}$. Then
\begin{displaymath}
  \mathbb{P}(Y_n > t) = \mathbb{P}(\xi_1 > t/n,\dots,\xi_n > t/n).
\end{displaymath}
By independence 
\begin{displaymath}
  \mathbb{P}(Y_n > t) = \prod_{j=1}^n \bigl(1-F_j(t/n)\bigr).
\end{displaymath}
Therefore, using hypothesis \eqref{eqn:first-order-dist}, 
\begin{displaymath}
  \lim_{n \to \infty} \mathbb{P}(Y_n > t) = e^{-t},
\end{displaymath}
proving the claim.

We then remark that the exact same proof holds if we only consider subsets of the sample points (e.g. $\{ x_i \in C\}$). This, along with independence of the $\xi_i$ implies that, for $f_n$ as in Theorem \ref{thm:main2}, we have that $nf_n \to W_\lambda + g$, where this convergence is in $\mathcal{P}(\mathcal{S})$.

The formula for the limiting joint distribution of the argmins then follows as in the proof of Lemma \ref{thm:main1}.

\end{proof}

\begin{rem}
 We remark that convergence in $\mathcal{P}(\mathcal{S})$ does not necessarily imply convergence in distribution of the sets $M_k$ ( except for the case $k=1$). This is because small scale features may be lost in the limiting procedure at the level of $k$-argmins. However, the ``pointwise'' construction that we use here preserves the convergence of $k$-argmins, since all argmins are locally isolated.
\end{rem}

\section{Computational example on $D=(0,1)$}\label{sec:comp-ex}

We illustrate some of our results by considering a family of distributions obtained from computing our ``stochastic integrals'' on the unit interval $[0,1]$.  Consider the function
\[ g(x):= x^2.\]

Obviously the distribution of the minimizers of $g$ is the Dirac delta measure at zero. Let us now consider the random function
\[ x \mapsto \delta W_\lambda(x)+ g(x),\]
for some parameter $\delta>0$, and for $\lambda \equiv 1$. The distribution of the minimizers of $\delta W_\lambda + g$ will naturally depend on the value of $\delta$ and is expected to be close to a Dirac delta measure for small values of $\delta$, whereas is expected to be essentially uniform for large values of $\delta$. We will then interpret the distributions of minimizers of $\delta W_\lambda + g$ as regularizations of the Dirac delta at zero.

We can compute explicitly the densities for the distribution of minimizers using \eqref{eqn:argmin-dist}. Indeed the distribution of the minimizer of $\delta W_\lambda + g$ is given by
\[   \rho_\delta(y) :=  \frac{2}{\delta}\int_{y}^1 x \exp \left( - \frac{2x^3}{3\delta} \right) dx  + \exp\left(-\frac{2\delta}{3} \right), \quad y \in [0,1].    \] 
Figure \ref{fig:Densities} plots this density for several representative values of $\delta$. 
\begin{figure}
  {\centering \includegraphics[width=.8\textwidth]{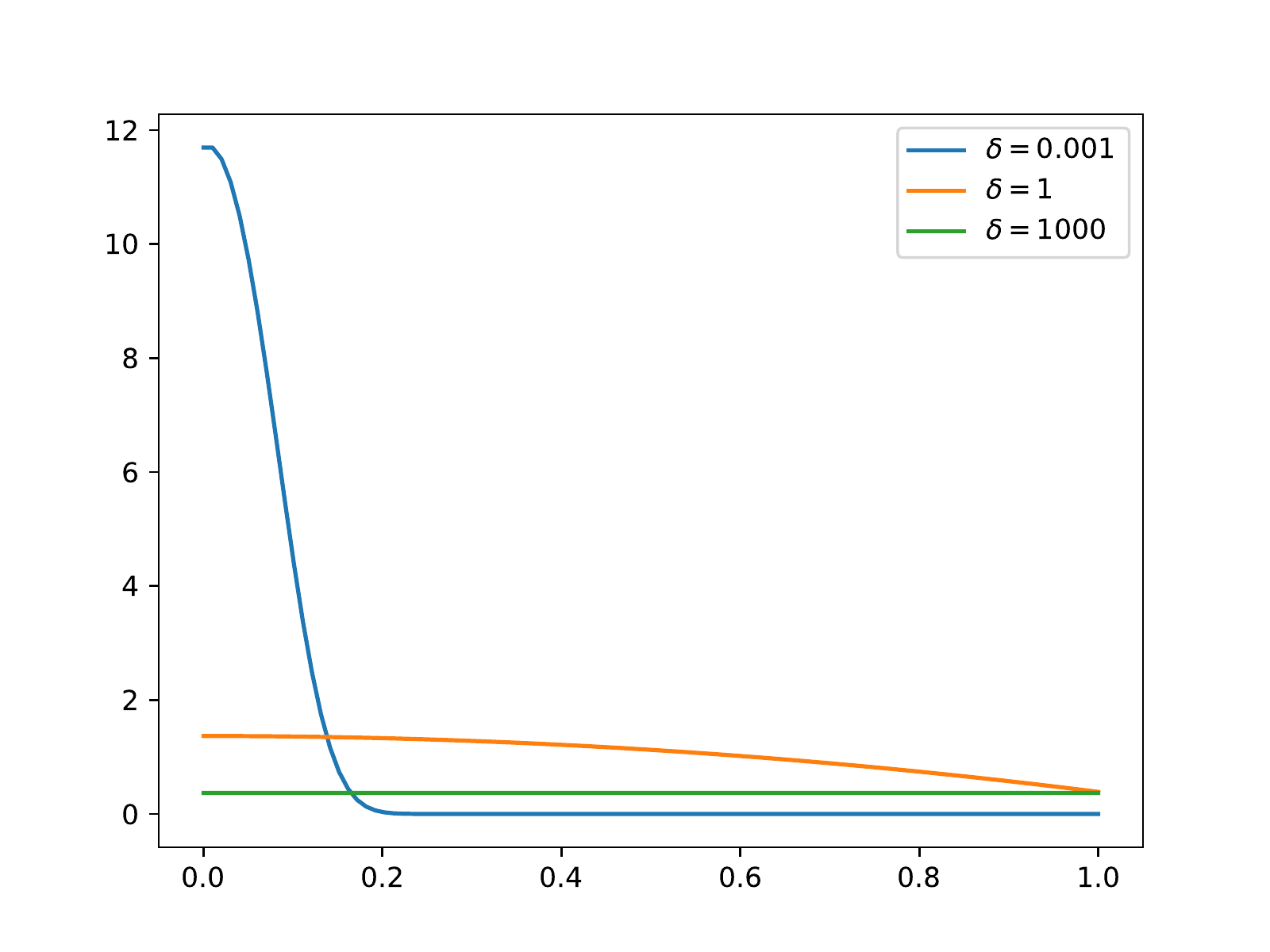}
  \caption{The distribution of minimizers of $x^2 + \delta W_\lambda$, for different values of $\delta$.} \label{fig:Densities}}
\end{figure}

Our results give us information of \textit{where} the minimizer is located. In particular it allows us to quantify the likelihood to find a minimizer of the process in any given subregion of $D$.

\appendix

  \bibliographystyle{alpha}
\bibliography{expbib}

\end{document}